\documentclass{mcom-l}

\newtheorem{theorem}{Theorem}[section]
\newtheorem{lemma}[theorem]{Lemma}
\newtheorem{corollary}[theorem]{Corollary}

\theoremstyle{definition}
\newtheorem{definition}[theorem]{Definition}
\newtheorem{problem}[theorem]{Problem}

\theoremstyle{remark}
\newtheorem{remark}[theorem]{Remark}

\numberwithin{equation}{section}

\usepackage{amsfonts}
\usepackage{inconsolata}
\usepackage[T1]{fontenc}
\usepackage{mathrsfs}
\usepackage{xcolor}
\usepackage{graphicx}
\usepackage{hyperref}
\usepackage{algorithmicx,algorithm}
\usepackage[noend]{algpseudocode}
\usepackage[caption = false]{subfig}
\usepackage{booktabs}
\usepackage{enumitem}
\begin{document}

\title{When Lanczos Iterations Generate Symmetric Quadrature Nodes?}


\author{Wenhao Li}
\address{Hong Kong Baptist University, Hong Kong, P.R. China}
\curraddr{}
\email{liwenhao@uic.edu.cn}
\thanks{Hong Kong Baptist University, Hong Kong, P.R. China; Guangdong Provincial Key Laboratory of Interdisciplinary Research and Application for Data Science, BNU-HKBU United International College, Zhuhai, 519087, P.R. China. }

\author{Zongyuan Han}
\address{Beijing Normal University, Beijing 100875, P.R. China.}
\curraddr{}
\email{202131130030@mail.bnu.edu.cn}
\thanks{Laboratory of Mathematics and Complex Systems, Ministry of Education;  School of Mathematical Sciences, Beijing Normal University, Beijing 100875, P.R. China.}

\author{Shengxin Zhu}
\address{Research Center for Mathematics, Beijing Normal University, Zhuhai 519087, P.R. China.}
\curraddr{}
\email{Shengxin.Zhu@bnu.edu.cn}
\thanks{Research Center for Mathematics, Beijing Normal University, Zhuhai 519087, P.R. China; Guangdong Provincial Key Laboratory of Interdisciplinary Research and Application for Data Science, BNU-HKBU United International College, Zhuhai 519087, P.R. China. }


\date{\today}


\begin{abstract}

The Golub-Welsch algorithm [\href{https://www.ams.org/journals/mcom/1969-23-106/S0025-5718-69-99647-1/}{ Math. Comp., 23: 221-230 (1969)}] has long been assumed symmetric for estimating quadratic forms. Recent research indicates that asymmetric quadrature nodes may be more often and the existence of a practical symmetric quadrature for estimating matrix quadratic form is even doubtful.
This paper derives a sufficient condition for symmetric quadrature nodes for estimating quadratic forms involving the Jordan-Wielandt matrices which frequently arise from many applications. The condition is closely related to how to construct an initial vector for the underlying Lanczos process. 
Applications of such constructive results are demonstrated by estimating the Estrada index in complex network analysis. 


\end{abstract}

\maketitle

\section{Introduction}
\label{sec:background}

Given a smooth function $f$ defined on $[a, b]$ and a non-negative weight function $\omega$, the Golub-Welsch algorithm \cite{GW69} produces a Gaussian quadrature rule,
$$\int_{a}^{b} f(t) \omega(t) dt \approx \sum_{i=1}^m \tau_k f(\theta_k), $$
where quadrature nodes $\{ \theta_k\}_{k=1}^m$ are the eigenvalues of the tridiagonal Jacobi matrix produced in the Lanczos iteration \cite{L50}, and the quadrature weights are the squares of the first element of each eigenvector of the Jacobi matrix. Such a Gauss quadrature rule based on the Lanczos method (or Lanczos quadrature method) plays a central role in estimating quadratic forms involving matrix functions  
\begin{equation}
      \label{eq:GQ}
    Q(\boldsymbol{u},f,A) = \boldsymbol{u}^T f(A) \boldsymbol{u} \approx \Arrowvert\boldsymbol{u}\Arrowvert_2^2 \sum_{k=1}^m \tau_k f(\theta_k),
\end{equation}
where $A$ is a symmetric matrix, $\boldsymbol{u}$ is a vector, see \cite{GS94,GM94} or the following sessions for more details. This method was previously used in the analysis of iterative methods \cite{F92,M97,BG99,CGR99} \cite[p.195]{M99} and the finite element method \cite{R92}. Bai, Golub and Fahey applied this method within the realm of Quantum Chromodynamics (QCD) \cite{BFG96}, and increasing more applications are emerging. In particular, it can be used to measure the centrality of complex networks. For example, subgraph centrality can be quantified as $\boldsymbol{e}_i^T e^{\beta A} \boldsymbol{e}_i$ \cite{ER05}, while resolvent-based subgraph centrality is associated with $\boldsymbol{e}_i^T(I - \alpha A)^{-1} \boldsymbol{e}_i$ \cite{EH10}. Fast computing such centrality is important in practical scenarios such as urban public transport \cite{BS21,BS22}. In genomics, statisticians often require approximations of the distribution of quadratic forms \eqref{eq:GQ} with normal distributed vectors \cite{LBPNR18,CL19,BQF}. In Gaussian process regression, hyperparameter optimization is achieved through maximization of a likelihood function, which requires the estimation of quadratic forms \eqref{eq:GQ} \cite{DE17,zhu18,SA22}. 
The Lanczos quadrature method offers effective computational means for accurately computing these quantities. Furthermore, this method can be integrated into Krylov spectral methods to solve time-dependent partial differential equations \cite{L05,L07}.

For a long time, the quadrature nodes generated by the Lanczos algorithm were assumed to be symmetric \cite{R69}. For example, the error analysis given by Ubaru, Chen and Saad relied upon such an assumption \cite[Section 4.1]{UCS17}. Until very recently, Cortinovis and Kressner highlighted that the quadrature nodes may more frequently exhibit asymmetry, consequently deriving a corresponding error analysis \cite{CK21}. Some researches even doubt the existence of symmetric quadrature for estimating quadratic forms. To make it clear, an intriguing inquiry arises regarding the circumstances under which the Lanczos algorithm would yield symmetric quadrature nodes for real applications. This question is equivalent to the following problem.
\begin{problem}
\label{prob:1}
Given a symmetric matrix $A$ and an initial vector $\boldsymbol{u}$, when are the Ritz values (eigenvalues of the tridiagonal Jacobi matrix obtained in the Lanczos process) symmetrically distributed?
\end{problem}

As far as we know, studies have investigated spatial distribution \cite{C11,CE12,B13,M15}, convergence characteristics \cite{TM12,ELM21}, and stabilization techniques \cite{MS06,P71,P80} regarding Ritz values within the Lanczos method and the Arnoldi process, but none of them directly address this problem. To answer Problem \ref{prob:1}, a series of interesting questions are proposed and answered step by step. Finally,  we delicately construct starting vectors for the Lanczos process such that symmetric quadratic rules are ensured for a wide class of matrices. Besides, the application of our theory is demonstrated by designed matrices and problems arising from network analysis.

The major contribution (Theorem \ref{Thm:main}) on guaranty of matrix-based symmetric quadrature rule for quadratic form estimations is enlightened by the scenario of symmetric zeros of Jacobi orthogonal polynomials. Other inspirations are from  the well-known Cauchy interlacing theorem and the property of symmetric eigenvalues with respect to Jordan-Wielandt matrices. In parallel to the symmetric weight functions for Jacobi orthogonal polynomials, we first introduce the \textit{absolute palindrome} (Definition \ref{Def:sav}) and show the existence of measure enjoying such a property (Theorem \ref{thm:ap}). Based on these, we derive a general intermediate result on when $m$-step Lanczos iteration generates Jacobi matrix with constant diagonal elements (Lemma \ref{Lem:Suf}) and then when symmetric quadrature rule may arise (Theorem \ref{Thm:main}).
 In particular, we show that one can construct a certain type of initial vectors for Jordan-Wielandt matrices (Theorem \ref{prop:1} and Theorem \ref{prop:2}) to guarantee symmetric Ritz values, i.e. symmetric quadrature rule. 
Besides, this paper also analyzes the lower and upper bounds of the difference between the two least required Lanczos iterations $m$ (for symmetric/asymmetric quadrature rule) that guarantee the same approximation accuracy of quadratic forms \eqref{eq:GQ} (Corollary \ref{coro:diffm}). 

The paper is organized as follows. We commence by revisiting the Lanczos quadrature method's process for estimating quadratic forms \eqref{eq:GQ} in Section \ref{sec:basics}. In Section \ref{sec:EA}, we introduce a sufficient condition (Theorem \ref{Thm:main}) tailored for the Lanczos quadrature with symmetric quadrature nodes and suggest the construction of initial vectors for Jordan-Wielandt matrices through Theorem \ref{prop:1} and Theorem \ref{prop:2}. Subsequently, in Section \ref{sec:comparison}, Corollary \ref{coro:diffm} studies the potential savings in iterations through such symmetry. Numerical experiments are given in Section \ref{sec:experiments} to illustrate the validity of our theoretical results in real applications such as complex network analysis on directed graphs, followed by conclusions in Section \ref{sec:cr}.

\section{The Lanczos quadrature method}
\label{sec:basics}
For a symmetric (positive definite\footnote{The Lanczos process does not require positive definiteness, but in some applications such property is demanded.}) matrix $A$, its quadratic form $Q(\boldsymbol{u},f,A)$ \eqref{eq:GQ} with vector $\boldsymbol{u}$ and matrix function $f$ can be estimated by the Lanczos quadrature method \cite{GS94,GM94}. With the eigen-decomposition $A = Q\Lambda Q^T$, one obtains $f(A) = Qf(\Lambda)Q^T$. Further, one may assume $\Lambda = \mathrm{diag}(\lambda_1, \lambda_2, \cdots, \lambda_n)$ and $\lambda_{\min} = \lambda_1 \le \lambda_2 \le \ldots \le \lambda_n = \lambda_{\max}$ without loss of generality. Then $f(\Lambda)$ is a diagonal matrix with entries $\{f(\lambda_j)\}_{j=1}^n$. Let $\boldsymbol{u}$ be normalized as $\boldsymbol{v} = \boldsymbol{u}/\Arrowvert\boldsymbol{u}\Arrowvert_2$, and $\boldsymbol{\mu} = Q^T \boldsymbol{v}$, \eqref{eq:GQ} further reads
\begin{equation}
    \label{eq:qfmu}
    Q(\boldsymbol{u},f,A) = \boldsymbol{u}^Tf(A)\boldsymbol{u} = \Arrowvert\boldsymbol{u}\Arrowvert_2^2 \boldsymbol{v}^T Q f(\Lambda)Q^T\boldsymbol{v} = \Arrowvert\boldsymbol{u}\Arrowvert_2^2 \boldsymbol{\mu}^Tf(\Lambda)\boldsymbol{\mu}. 
\end{equation}
The last quadratic form in \eqref{eq:qfmu} can be reformulated as a Riemann-Stieltjes integral $\mathcal{I}$, 
\begin{equation}
    \begin{aligned}
    \label{eq:RS}
    \Arrowvert\boldsymbol{u}\Arrowvert_2^2 \boldsymbol{\mu}^T f(\Lambda)\boldsymbol{\mu} = \Arrowvert\boldsymbol{u}\Arrowvert_2^2\sum_{j=1}^n f(\lambda_j)\mu_j^2  = \Arrowvert\boldsymbol{u}\Arrowvert_2^2 \int_{\lambda_1}^{\lambda_n}f(t) d\mu(t) = \Arrowvert\boldsymbol{u}\Arrowvert_2^2 \mathcal{I}.
\end{aligned}
\end{equation}
where $\mu_j$ is the $j^{th}$ element of $\boldsymbol{\mu}$ and measure $\mu(t)$ is a piecewise constant function of $\mathcal{I}$
\begin{equation}
    \label{eq:measure_f}
        \mu(t)=\left\{
    \begin{aligned}
    &\quad 0 , & {\rm if} \ t < \lambda_1=a, \\
    &\sum_{j=1}^{k-1} \mu_j^2 , & {\rm if}\ \lambda_{k-1} \le t < \lambda_k, k=2,...,n, \\
    &\sum_{j=1}^{n} \mu_j^2 , & {\rm if} \ t \ge \lambda_n=b.
    \end{aligned}
    \right.
\end{equation}
According to the Gauss quadrature rule \cite[Chapter 6.2]{GM09}, the Riemann Stieltjes integral $\mathcal{I}$ can be approximated by an $m$-point Lanczos quadrature rule $\mathcal{I}_{m}$ so the last term in equation \eqref{eq:RS} reads
\begin{equation}
\label{eq:GQR}
\Arrowvert\boldsymbol{u}\Arrowvert_2^2 \mathcal{I} \approx \Arrowvert\boldsymbol{u}\Arrowvert_2^2 \mathcal{I}_m  = \Arrowvert\boldsymbol{u}\Arrowvert_2^2 \sum_{k=1}^{m} \tau_k f(\theta_k) \equiv Q_m(\boldsymbol{u},f,A),
\end{equation}
where $m$ is the number of quadrature nodes. The quadrature nodes $\{\theta_k\}_{k=1}^m$ and weights $\{\tau_k\}_{k=1}^m$ can be obtained by the Lanczos algorithm \cite{BFG96,GM09}. Let $T_m$ be the Jacobi matrix obtained in the Lanczos process, then the nodes are the eigenvalues of $T_m$, and weights are the squares of the first elements of the corresponding normalized eigenvectors \cite{GW69}. Algorithm \ref{alg:Lanc} outlines how to compute the quadratic form \eqref{eq:GQ} via the Lanczos quadrature method \cite[Section 7.2]{GM09}. Note that if the Lanczos process breaks down before or at step $m$, Algorithm \ref{alg:Lanc} computes \eqref{eq:GQ} exactly and terminates.
\begin{algorithm}[t]
        \raggedright
	\caption{Lanczos Quadrature Method for Quadratic Form Estimation} 
	\label{alg:Lanc}
	\hspace*{0.02in} {\bf Input:} Symmetric (positive definite) matrix $A \in \mathbb{R}^{n\times n}$, vector $\boldsymbol{u} \in \mathbb{R}^{n}$, matrix function $f$, steps of Lanczos iterations $m$.\\
	\hspace*{0.02in} {\bf Output:} Approximation of the quadratic form $Q_m(\boldsymbol{u}, f, A) \approx \boldsymbol{u}^T f(A) \boldsymbol{u}$.
	\begin{algorithmic}[1]
	    \State $\boldsymbol{v}_{1} = \boldsymbol{v} = \boldsymbol{u}/\Arrowvert \boldsymbol{u}\Arrowvert_2$
        \State $\alpha_1 = {\boldsymbol{v}_{1}}^T A \boldsymbol{v}_{1}$
	    \State $\boldsymbol{u}_{2} = A\boldsymbol{v}_{1} - \alpha_1 \boldsymbol{v}_{1}$
		\For{$k = 2 \text{ to } m$}
			\State $\beta_{k-1} = \Arrowvert\boldsymbol{u}_k\Arrowvert_2$
			\If{$\beta_{k-1} = 0$} 
                \State $m = k-1$
                \State \textbf{break}
                \EndIf 
                \State $\boldsymbol{v}_k = \boldsymbol{u}_k/\beta_{k-1}$
			\State $\alpha_{k} = {\boldsymbol{v}_{k}}^T A \boldsymbol{v}_{k}$
			\State $\boldsymbol{u}_{k+1} = A \boldsymbol{v}_{k} - \alpha_{k}             \boldsymbol{v}_{k} - \beta_{k-1}\boldsymbol{v}_{k-1}$
		\EndFor
            \State \textbf{end for}
            \State $T_{m} = \left[\begin{array}{ccccc}
               \alpha_1 &  \beta_1 & 0 & \cdots & 0 \\
               \beta_1 & \alpha_2 & \beta_2 & \ddots & \vdots \\
               0 & \beta_2 & \ddots & \ddots & 0 \\
               \vdots & \ddots & \ddots & \ddots & \beta_{m-1} \\
               0 & \cdots & 0 & \beta_{m-1} & \alpha_{m}
            \end{array}\right].$
            \State $[V,D] = \texttt{eig}(T_{m})$
            \State $[\tau_1,\ldots,\tau_m] = (\boldsymbol{e}_1^T V)\odot(\boldsymbol{e}_1^T V) $
            \State $[\theta_1,\ldots,\theta_m]^T = \texttt{diag}(D)$
            \State \Return $Q_m(\boldsymbol{u},f,A) = \Arrowvert\boldsymbol{u}\Arrowvert_2^2 \sum_{k=1}^m \tau_k f(\theta_k)$
	\end{algorithmic}
\end{algorithm}

\section{Symmetric Ritz values in the Lanczos iterations}
\label{sec:EA}

In this section we address the inquiry raised in Problem \ref{prob:1}. Recall that given a symmetric matrix $A$ and a starting vector $\boldsymbol{u}$, Algorithm \ref{alg:Lanc}, executed for $m$ steps, yields a symmetric tridiagonal matrix denoted as $T_{m}$.
We are enlightened by the well-known Cauchy interlacing theorem  \cite[Theorem 10.1.1]{P98}.
\begin{theorem}
    \label{lem:cauchy} 
    Let $T_m \in \mathbb{R}^{m\times m}$ be the symmetric tridiagonal matrix generated in the $m$-step Lanczos method with symmetric matrix $A \in \mathbb{R}^{n \times n}$ and initial vector $\boldsymbol{u}$. Then under the assumption that the values are non-decreasing by index, the eigenvalues of $A$, $\{\lambda_{k}\}_{k=1}^{n}$, interlace the eigenvalues of $T_{m}$, $\{\theta_k\}_{k=1}^{m}$,
    \begin{equation}
        \label{eq:cauchy}
        \lambda_k \le \theta_k \le \lambda_{k+n-m}, k = 1,\ldots, m.
    \end{equation}
\end{theorem}
\begin{proof}
    Two cases should be considered in the discussion of Lanczos iterations. First, if the Lanczos process does not break down at any step $m < n$, then all the eigenvalues of the principal submatrix (of $T_n$) $T_m,m<n$ are simple. Moreover, since $T_n = V_n^T A V_n$, where $V_n$ is the orthogonal matrix that stores the vectors $\{\boldsymbol{v}_k\}_{k=1}^n$ generated during the iteration, the symmetric tridiagonal matrix $T_n$ has the same eigenvalues as $A$'s. Then according to \cite[Theorem 10.1.1]{P98}, one may trivially obtain \eqref{eq:cauchy} and no equal sign is possible. On the other hand, if the Lanczos process breaks down at certain step $m < n$, then it cannot be run to completion and $T_m$ has the same eigenvalues as $A$'s \cite[p. 41]{GM09} and all the eigenvalues of smaller Jacobi matrices $T_i, i<m$ are strictly interlaced by them \cite[Theorem 3.3]{GM09}. In such scenario, the equal signs in \eqref{eq:cauchy} hold.
\end{proof}

One may observe that if matrix $A$ possesses an extremely skewed eigenvalue distribution, the resulting Ritz values are unlikely to be always symmetric. Consequently, we assume the symmetry of the eigenvalue distribution as a crucial property for ensuring symmetric Ritz values during the Lanczos iteration. This raises the question:
\begin{problem}
    \label{prob:mat}
    What types of symmetric matrices $A$ have symmetric eigenvalues?
\end{problem}

In Section \ref{sec:mat} we discuss a type of matrices characterized by the symmetry of eigenvalues, which arises in various practical applications.

\subsection{Symmetric matrices with symmetric eigenvalues}
\label{sec:mat}

There are several works demonstrating that the Jordan-Wielandt matrices of the following form
\begin{equation}
    \label{Eq:A}
    A = \begin{bmatrix}
        O & B \\
        B^H & O
    \end{bmatrix} \in \mathbb{C}^{(n_1+n_2) \times (n_1+n_2)},
\end{equation}
have symmetric eigenvalue distributions \cite[Theorem 1.2.2]{B96}\cite[Proposition 4.12]{S03}, where $O$ denotes zero matrix and $B \in \mathbb{C}^{n_1 \times n_2}$. 

Such matrices \eqref{Eq:A} exist in various applications. For instance, the graph of a finite difference matrix is \emph{bipartite}, meaning that the vertices can be divided into two sets by the red-black order so that no edges exist in each set \cite[p. 211]{HY81} \cite[p. 123]{S03}. Such matrices can be transformed into the form of \eqref{Eq:A} with non-zero diagonal entries. In the analysis of complex networks, a directed network of $n$ nodes with asymmetric adjacency matrix $B \in \mathbb{R}^{n\times n}$ can be extended to a bipartite undirected network with symmetric \emph{block supra-adjacency matrix} of form \eqref{Eq:A} via bipartization \cite{BEK13,BB20,BS22}. Besides, given a matrix $B \in \mathbb{R}^{n_1\times n_2}$ and two vectors $\boldsymbol{x} \in \mathbb{R}^{n_1}, \boldsymbol{y} \in \mathbb{R}^{n_2}$, a bilinear form $\boldsymbol{x}^T B \boldsymbol{y}$ can be transformed to a quadratic form
$\boldsymbol{z}^T A \boldsymbol{z}$ with $$\boldsymbol{z} = \begin{bmatrix}
    \boldsymbol{x} \\
    \boldsymbol{y}
\end{bmatrix} \text{ and } A = \frac{1}{2} \begin{bmatrix}
    O & B \\
    B^T & O
\end{bmatrix}$$
since
$$ \begin{aligned}
    \boldsymbol{z}^T A \boldsymbol{z} = \frac{1}{2} \begin{bmatrix}
    \boldsymbol{x}^T & \boldsymbol{y}^T
\end{bmatrix} \begin{bmatrix}
    O & B \\
    B^T & O
\end{bmatrix} \begin{bmatrix}
    \boldsymbol{x} \\
    \boldsymbol{y}
\end{bmatrix}
& = \frac{\boldsymbol{y}^TB^T\boldsymbol{x} + \boldsymbol{x}^T B \boldsymbol{y}}{2}& = \boldsymbol{x}^T B \boldsymbol{y}.
\end{aligned}$$

Such Jordan-Wielandt matrices enjoys the property of symmetric distribution of eigenvalues.
\begin{theorem}
\label{thm:m=n} \cite[Theorem 1.2.2]{B96}
    Let the singular value decomposition of $B \in \mathbb{C}^{n_1 \times n_2}$ be $B = U \Sigma V^H,$ where $\Sigma = \begin{bmatrix}
        \Sigma_1 & O \\ O & O
    \end{bmatrix}$, $\Sigma_1 = \mathrm{diag}(\sigma_1, \sigma_2, \cdots, \sigma_r)$, $r = \mathrm{rank}(B)$,
    $$U = \left[U_1, U_2\right], U_1 \in \mathbb{C}^{n_1 \times r}, U_2 \in \mathbb{C}^{n_1 \times (n_1-r)},$$
    $$V = \left[V_1, V_2\right], V_1 \in \mathbb{C}^{n_2 \times r}, V_2 \in \mathbb{C}^{n_2 \times (n_2-r)}.$$
    Then
    $$A = \begin{bmatrix}
        O & B \\ B^H & O
    \end{bmatrix} = P^H \begin{bmatrix}
        \Sigma_1 & O & O & O \\ O &  -\Sigma_1 & O & O \\ O & O & O & O \\O & O & O & O
    \end{bmatrix} P,$$
    where $P$ is unitary
    $$ P = \frac{1}{\sqrt{2}}\begin{bmatrix}
        U_1 & U_1 & \sqrt{2} U_2 & O \\ V_1 & -V_1 & O & \sqrt{2}V_2 
    \end{bmatrix}^H.$$
    Thus $2r$ eigenpairs of $A$ are $\left(\pm \sigma_i, \begin{bmatrix}
        \boldsymbol{u}_i \\
        \pm \boldsymbol{v}_i
    \end{bmatrix}\right), i = 1,\ldots,r$, and the left eigenvalues $0$ repeated $(n_1+n_2-2r)$ times, where $\boldsymbol{u}_i$ and $\boldsymbol{v}_i$ are the columns of $U_1$ and $V_1$ respectively.
\end{theorem}

It is trivial that if the diagonal entries of a Jordan-Wielandt matrix are $\gamma$, then $2r$ eigenpairs of $A$ are $\left(\gamma \pm \sigma_i, \begin{bmatrix}
    \boldsymbol{u}_i \\ \pm \boldsymbol{v}_i
\end{bmatrix}\right), i = 1,\ldots, r$. Therefore, we find a class of matrices with symmetric eigenvalue distributions arising from various applications. Nonetheless, presuming a symmetric eigenvalue distribution may not suffice to ensure the symmetry of Ritz values. We proceed to inquire:
\begin{problem}
\label{prob:2}
Given a symmetric matrix $A$ with symmetric eigenvalues and an initial vector $\boldsymbol{u}$, when are the Ritz values (eigenvalues of the tridiagonal Jacobi matrix obtained in the Lanczos process) symmetrically distributed?
\end{problem}

Consider the case of $n$-degree Jacobi orthogonal polynomial 
$$P_{n}^{(\alpha,\beta)}(t) = \frac{(-1)^n}{2^n n!}(1-t)^{-\alpha}(1+t)^{-\beta} \frac{d^n}{dt^n}\left[ (1-t)^{n+\alpha}(1+t)^{n + \beta} \right]$$
with the continuous weight function $\omega(t) = (1-t)^{\alpha}(1+t)^{\beta},t\in[-1,1]$. When $\alpha = \beta$, the weight function $\omega(t)$ is even, the measure function $\mu(t)$ is central symmetric and the zeros of $P_{n}^{(\alpha,\beta)}(t)$ are symmetric about $0$. Enlightened by such an observation, we conjecture that the symmetry of Ritz values generated by the Lanczos process is linked to the underlying discrete measure function $\mu(t)$ as outlined in \eqref{eq:measure_f}. Elaborate explanations and discussions regarding this conjecture are provided in Section \ref{sec:ap}.

\subsection{Central symmetric discrete measure functions}
\label{sec:ap}

As discussed, the central symmetry of $\mu(t)$ is considered as another essential property for guaranteeing the symmetry of Ritz values in Lanczos iterations. One may observe that \eqref{eq:measure_f} is defined by the components of $\boldsymbol{\mu} = Q^T \boldsymbol{v}$, where $Q$ denotes the orthonormal matrix in the eigen-decomposition of $A = Q \Lambda Q^T$, and $\boldsymbol{v} = \boldsymbol{u}/\Arrowvert\boldsymbol{u}\Arrowvert_2$ represents the normalized initial vector. The term $\mu_i^2$ signifies the increment of $\mu(t)$ from $t = \lambda_i$ to $t = \lambda_{i+1}$, where $\mu_i$ denotes the $i^{th}$ element of $\boldsymbol{\mu}$. Building upon our conjecture, we postulate that the first and last $i^{th}$ step increments $\mu_i^2$ of $\mu(t)$ are equal for $i = 1,\ldots,\lfloor n/2 \rfloor$. The subsequent definition aids in elucidating this assumption.
\begin{definition}
    \label{Def:sav}
    A vector $\boldsymbol{w} \in \mathbb{R}^n$ is said to be an \textit{absolute palindrome} if 
    $$ |w_i| = |w_{n+1-i}|, i = 1,\ldots, \lfloor\frac{n}{2}\rfloor. $$
\end{definition}

Similar to the question raised in Problem \ref{prob:mat}, we would like to know,
\begin{problem}
    \label{prob:vec}
    Given a symmetric matrix $A = Q\Lambda Q^T$ and a normalized initial vector $\boldsymbol{v}$. What conditions will guarantee an absolute palindrome $\boldsymbol{\mu} = Q^T \boldsymbol{v}$?
\end{problem}

Theorem \ref{thm:ap} establishes the existence of such conditions.
\begin{theorem}
    \label{thm:ap}
    Let $Q = [\boldsymbol{q}_1,\ldots,\boldsymbol{q}_n]\in \mathbb{R}^{n\times n}$ be an orthonormal matrix, $\boldsymbol{\xi} = [\xi_1,\ldots,\xi_n]^T$ be a coefficient vector and $\boldsymbol{v} \in \mathbb{R}^n$ be a normalized initial vector for the Lanczos iteration. Then $\boldsymbol{v}$ can be written as a linear combination
    $$ \boldsymbol{v} = Q \boldsymbol{\xi} =  \sum_{i=1}^n \xi_i \boldsymbol{q}_i$$
    with the columns of $Q$ as the basis of $\mathbb{R}^n$ and real coefficients $\{\xi_i\}_{i=1}^n$. If $\boldsymbol{\xi}$ is an absolute palindrome, then $\boldsymbol{\mu} = Q^T \boldsymbol{v}$ is also an absolute palindrome.
\end{theorem}
\begin{proof}
    First, $\boldsymbol{\mu}$ reads
$$ \boldsymbol{\mu} = Q^T \boldsymbol{v} = \begin{bmatrix}
    \boldsymbol{q}_1^T \\
    \boldsymbol{q}_2^T \\
    \vdots \\
    \boldsymbol{q}_n^T
\end{bmatrix} \left[\xi_1 \boldsymbol{q}_1 + \xi_2\boldsymbol{q}_2 + \cdots + \xi_n \boldsymbol{q}_n \right] = \begin{bmatrix}
    \xi_1 \\
    \xi_2 \\
    \vdots \\
    \xi_n
\end{bmatrix} = \boldsymbol{\xi}.$$
This indicates that the assumption of $|\mu_{i}| = |\mu_{n+1-i}|$ is equivalent to $|\xi_i| = |\xi_{n+1-i}|, \\ i = 1,\ldots,\lfloor \frac{n}{2} \rfloor$. There must exist certain vector $\boldsymbol{\xi}$ that guarantees the property. Moreover, suppose $n$ is even, for vector $\boldsymbol{\xi}$, it can be expressed as
$$ \boldsymbol{\xi} = \xi_1 \begin{bmatrix}
    1 \\ 0 \\ \vdots \\ 0 \\ 0 \\ \vdots \\ 0 \\ \pm 1 
\end{bmatrix} + \xi_2 \begin{bmatrix}
    0 \\ 1 \\ \vdots \\ 0 \\ 0 \\ \vdots \\ \pm 1 \\ 0
\end{bmatrix} + \cdots + \xi_{\frac{n}{2}} \begin{bmatrix}
    0 \\ 0 \\ \vdots \\ 1 \\ \pm 1 \\ \vdots \\ 0 \\ 0
\end{bmatrix} = \xi_1 \boldsymbol{s}_1^{\pm} + \xi_2 \boldsymbol{s}_2^{\pm} + \cdots \xi_{\frac{n}{2}} \boldsymbol{s}_{\frac{n}{2}}^{\pm}. $$
Such independent vectors $\{\boldsymbol{s}_{j}^{\pm}\}_{j=1}^{n/2}$ form a subspace $S = \mathrm{span}\{\boldsymbol{s}_1^{\pm}, \boldsymbol{s}_2^{\pm}, \ldots, \boldsymbol{s}_{\frac{n}{2}}^{\pm}\}$, where the dimension is $n/2$. Note that there are $2^{\frac{n}{2}}$ such subspaces. For odd $n$, there would be $2^{\frac{n}{2} + 1}$ subspaces due to the additional vector $\boldsymbol{s}_{\frac{n}{2} + 1}^{\pm} = (0,\cdots,0,\pm 1,0,\cdots,0)^T$.
\end{proof}

\begin{remark}
    Note that the property of absolute palindrome may not be preserved under orthogonal transformation. For example, let $\boldsymbol{v}=(1, 2, 3, 2, 1)^T,$
    $$Q=\begin{bmatrix}
        0 & 1 & 0 & 0 & 0 \\
        1 & 0 & 0 & 0 & 0 \\
        0 & 0 & 0 & 1 & 0 \\
        0 & 0 & 1 & 0 & 0 \\
        0 & 0 & 0 & 0 & 1
    \end{bmatrix},$$
    then $\boldsymbol{\mu}=Q^T \boldsymbol{v}= (2, 1, 2, 3, 1),$ which is not an absolute palindrome.
\end{remark}

Now we proceed to illustrate that the mentioned two conditions suffice to ensure symmetric Ritz values in the Lanczos process.

\subsection{A sufficient condition of symmetric Ritz values in the Lanczos iterations}
\label{Sec:existenceSQN}

We first prove that when $A$ has symmetric eigenvalue distribution and $\boldsymbol{\mu}^1 = Q^T \boldsymbol{v}^1$ is an absolute palindrome, the tridiagonal Jacobi matrix $T_{m}$ generated by the $m$-step Lanczos iteration has constant diagonal entries $\bar{\lambda}$.

\begin{lemma}
\label{Lem:Suf}
     Let $A = Q\Lambda Q^T \in \mathbb{R}^{n \times n}$ be symmetric, $\Lambda = \mathrm{diag}(\lambda_1,\lambda_2, \cdots,\lambda_n)$ with eigenvalues $\lambda_1\le \lambda_2 \le \ldots \le \lambda_n$ symmetric about $\bar{\lambda}$, and $\boldsymbol{v}^1 \in \mathbb{R}^{n}$ be a normalized initial vector for the Lanczos iteration. For the $m$-step Lanczos method with $A$ and $\boldsymbol{v}^1$ ($ m < n$), if vector $\boldsymbol{\mu}^1 = Q^T \boldsymbol{v}^1$ is an absolute palindrome, the Jacobi matrix $T_{m}$ generated by $m$-step Lanczos iteration will have constant diagonal entries $\bar{\lambda}$.
\end{lemma}
\begin{proof}
    We first recall that if one runs the $m$-step Lanczos iterations with $A - \rho I$ instead of $A$ with the same initial vector $\boldsymbol{v}^1$, then the same Lanczos vectors would be obtained and the symmetric tridiagonal matrix $T_m$ transforms as $T_m - \rho I_m$ \cite[Section 12.2.2]{P98}. Without loss of generality, we take $\rho = \bar{\lambda}$, let $A \to A - \rho I$, $\Lambda \to \Lambda - \rho I$ and prove $T_m \to T_m - \rho I_m$ having zero diagonals, $\alpha_k = 0, k = 1,\ldots,m$.

    Two important matrices $P$ and $S$ are introduced to complete this proof. Let $P$ denote the permutation matrix that reverse the order of entries, i.e., $(P \boldsymbol{\mu}^1)_i = \boldsymbol{\mu}^1_{n+1-i}, i = 1,\ldots,n$, and $S$ be the signature matrix (with diagonal entries $\pm 1$) that ensures $P\boldsymbol{\mu}^1 = S \boldsymbol{\mu}^1$. Note that $P$ and $S$ guarantee
    $$ P\Lambda P^T = -\Lambda, \Lambda S = S \Lambda, S = S^T, S^2 = I. $$
    For $k = 1,\ldots,m$, denote $\boldsymbol{\mu}^k = Q^T \boldsymbol{v}^k$, we wish to show
    \begin{equation}
        \label{eq:induct}
        \alpha_k = (\boldsymbol{v}^k)^T A \boldsymbol{v}^k =  0, \; P\boldsymbol{\mu}^{k+1} = (-1)^k S \boldsymbol{\mu}^{k+1}
    \end{equation}
    by mathematical induction. 
    
    Note that the second equality indicates that $\boldsymbol{\mu}^{k+1}$ is an absolute palindrome. For $k = 1$,
    $$ \begin{aligned}
        \alpha_1 &= (\boldsymbol{v}^1)^T Q \Lambda Q^T \boldsymbol{v}^1 =  (\boldsymbol{\mu}^1)^T \Lambda \boldsymbol{\mu}^1 = (P\boldsymbol{\mu}^1)^T P \Lambda P^T (P \boldsymbol{\mu}^1) = (S \boldsymbol{\mu}^1)^T (-\Lambda)(S \boldsymbol{\mu}^1) \\
        &= -(\boldsymbol{\mu}^1)^T S^T \Lambda S\boldsymbol{\mu}^1 = -(\boldsymbol{\mu}^1)^T S^T S \Lambda \boldsymbol{\mu}^1 = -(\boldsymbol{\mu}^1)^T \Lambda \boldsymbol{\mu}^1 = -\alpha_1.
    \end{aligned} $$
    Clearly $\alpha_1 = 0$. Together with the relationship between the first and second Lanczos vectors, we have
    $$\beta_1 \boldsymbol{v}^2 = A \boldsymbol{v}^1 - \alpha_1 \boldsymbol{v}^1 = A \boldsymbol{v}^1,  $$
    and thus
    $$ \boldsymbol{\mu}^2 = Q^T \boldsymbol{v}^2 = \frac{1}{\beta_1} Q^T Q \Lambda Q^T \boldsymbol{v}^1 = \frac{1}{\beta_1} \Lambda \boldsymbol{\mu}^1.$$
    Then
    $$ P \boldsymbol{\mu}^2 = \frac{1}{\beta_1} P \Lambda P^T P \boldsymbol{\mu}^1 = - \frac{1}{\beta_1} \Lambda P \boldsymbol{\mu}^1 = -\frac{1}{\beta_1} \Lambda S\boldsymbol{\mu}^1 = -S \left(\frac{1}{\beta_1} \Lambda \boldsymbol{\mu}^1\right) = -S \boldsymbol{\mu}^2, $$
    which proves the second equation in \eqref{eq:induct}. Now assume \eqref{eq:induct} is correct for $k\ge 1$. For $k + 1$,
    $$ \alpha_{k+1} = (\boldsymbol{\mu}^{k+1})^T \Lambda \boldsymbol{\mu}^{k+1} = (P \boldsymbol{\mu}^{k+1})^TP \Lambda P^T P \boldsymbol{\mu}^{k+1} = -(S\boldsymbol{\mu}^{k+1})^T \Lambda S\boldsymbol{\mu}^{k+1} = -\alpha_{k+1},$$
    which gives $\alpha_{k+1} = 0$. Based on the three-term recurrence between Lanczos vectors 
    $$ \beta_{k+1} \boldsymbol{v}^{k+2} = A \boldsymbol{v}^{k+1} - \alpha_{k+1}\boldsymbol{v}^{k+1} - \beta_k \boldsymbol{v}^k = A\boldsymbol{v}^{k+1} - \beta_k \boldsymbol{v}^k,$$ 
    it is trivial that
    $$ \boldsymbol{\mu}^{k+2} = Q^T \boldsymbol{v}^{k+2} = \frac{1}{\beta_{k+1}}\left( Q^T Q \Lambda Q^T \boldsymbol{v}^{k+1} - \beta_k Q^T \boldsymbol{v}^k \right) = \frac{1}{\beta_{k+1}} \left(\Lambda \boldsymbol{\mu}^{k+1} - \beta_k \boldsymbol{v}^{k}\right).$$
    Then 
    $$ \begin{aligned}
        P\boldsymbol{\mu}^{k+2} &= \frac{1}{\beta_{k+1}}\left(P\Lambda P^T P \boldsymbol{\mu}^{k+1} - \beta_k P \boldsymbol{\mu}^k\right)\\
        &= \frac{1}{\beta_{k+1}} \left( -\Lambda P \boldsymbol{\mu}^{k+1} - \beta_k P \boldsymbol{\mu}^{k} \right) \\
        &= \frac{1}{\beta_{k+1}} \left( -(-1)^k \Lambda S \boldsymbol{\mu}^{k+1} - (-1)^{k-1}\beta_k S \boldsymbol{\mu}^k \right) \\
        &= (-1)^{k+1} S \left(\frac{1}{\beta_{k+1}} \left( \Lambda \boldsymbol{\mu}^{k+1} - \beta_k \boldsymbol{\mu}^k \right) \right) \\
        & = (-1)^{k+1} S \boldsymbol{\mu}^{k+2},
    \end{aligned} $$
    which completes the proof of \eqref{eq:induct}. Lemma \ref{Lem:Suf} is proven by simply adding $\rho = \bar{\lambda}$ back to the transformed $A$.
\end{proof}

Under the same assumptions, the Lanczos process results in symmetrically distributed Ritz values. Meanwhile, the quadrature weights with respect to the pairs of symmetric quadrature nodes are equal.

\begin{theorem}
\label{Thm:main}
    Let $A = Q\Lambda Q^T \in \mathbb{R}^{n \times n}$ be symmetric, $\Lambda = \mathrm{diag}(\lambda_1,\lambda_2,\cdots,\lambda_n)$ with eigenvalues $\lambda_1\le \lambda_2 \le \ldots \le \lambda_n$ symmetric about $\bar{\lambda}$, and $\boldsymbol{v}^1 \in \mathbb{R}^{n}$ be a normalized initial vector for the Lanczos iteration. For the $m$-step Lanczos method with $A$ and $\boldsymbol{v}^1$ ($ m = 1,\ldots, n$), if vector $\boldsymbol{\mu}^1 = Q^T \boldsymbol{v}^1$ is an absolute palindrome, the distribution of $m$ Ritz values will be symmetric about $\bar{\lambda}$, and the quadrature weights corresponding to the pairs of symmetric quadrature nodes are equal.
\end{theorem}
\begin{proof}
    It is worth noting that every symmetric tridiagonal matrix with constant diagonal elements $\bar{\lambda}$ can be rearranged into the form
    \begin{equation}
    \label{Eq:JacobiIter}
        \begin{bmatrix}
            \bar{\lambda} I_{n_1} & B \\ B^T & \bar{\lambda} I_{n_2}
        \end{bmatrix}.
    \end{equation}
    This rearrangement follows the well-known red-black ordering \cite[p. 211]{HY81} \cite[p. 123]{S03}, where either $n_1 = n_2$ or $n_1 = n_2 + 1$. According to Theorem \ref{thm:m=n}, if the conditions specified in Theorem \ref{Thm:main} are satisfied, the Jacobi matrix $T_{m}$ yielded by the $m$-step Lanczos iteration will possess a symmetric eigenvalue distribution. Namely, the $m$ Ritz values exhibit symmetry about $\bar{\lambda}$. Furthermore, based on the equivalence of the pairs of first entries in the eigenvectors corresponding to the symmetric eigenvalues (as demonstrated in Theorem \ref{thm:m=n}), symmetrically equal quadrature weights are guaranteed for such Gaussian quadrature rules. The proof of Theorem \ref{Thm:main} is thereby concluded.
\end{proof}

\subsection{Construction of initial vectors for Jordan-Wielandt matrices}

\label{sec:construction}
From a practical point of view, it is of interest to ask:
\begin{problem}
    \label{prob:JW}
    Let $B \in \mathbb{R}^{n_1 \times n_2}$ and a Jordan-Wielandt matrix $A = Q\Lambda Q^T = \begin{bmatrix}
        O & B \\ B^T & O
    \end{bmatrix}$ have non-decreasing diagonal entries in $\Lambda$. Is there any initial vector $\boldsymbol{v}^1$ that guarantees measure vector $\boldsymbol{\mu}^1 = Q^T \boldsymbol{v}^1$ is an absolute palindrome?
\end{problem}

Problem \ref{prob:JW} would be studied in two scenarios: $n_1 = n_2$ and $n_1 > n_2$. In both cases $B$ is assumed to be of full (column) rank. Besides, note that the property of absolute palindrome can be expressed in the language of linear algebra with the aid of anti-diagonal matrix $I_n^{(anti)}$.

\begin{remark}
    If a vector $\boldsymbol{w} = \begin{bmatrix}
    \boldsymbol{w}_u \\ \boldsymbol{w}_d
\end{bmatrix} \in \mathbb{R}^{2n}$ is an absolute palindrome, then the vector $\boldsymbol{w}^{(abs)} \in \mathbb{R}^{2n}$ with absolute values of entries in $\boldsymbol{w}$ can be written as
\begin{equation}
\label{wabs}
    \boldsymbol{w}^{(abs)} = \begin{bmatrix}
    \boldsymbol{w}^{(abs)}_u \\
    I^{(anti)}_n \boldsymbol{w}^{(abs)}_u
\end{bmatrix}.
\end{equation}
\end{remark}

\subsubsection{Case 1: $n_1 = n_2$}
\label{sec:case1}
Assume $B$ is a square matrix of full rank, we propose that
\begin{theorem}
    \label{prop:1}
    Let $B \in \mathbb{R}^{n \times n}$ be of full rank and $A = \begin{bmatrix}
        O & B \\
        B^T & O
    \end{bmatrix} \in \mathbb{R}^{2n \times 2n}$ be a Jordan-Wielandt matrix with eigen-decomposition $Q^\dagger \Lambda^\dagger {Q^\dagger}^T$. Assume  $\Lambda^\dagger$ has non-decreasing diagonal entries. Then any initial vector that has either form
    $$ \boldsymbol{v}^1_u = \begin{bmatrix}
         \boldsymbol{v}_u \\
         \boldsymbol{0}
    \end{bmatrix} \quad \textit{or} \quad \boldsymbol{v}^1_d = \begin{bmatrix}
        \boldsymbol{0} \\
        \boldsymbol{v}_d
    \end{bmatrix} $$
    with real vectors $\boldsymbol{v}_u, \boldsymbol{v}_d \in \mathbb{R}^n$ and zero vector $\boldsymbol{0} \in \mathbb{R}^n$ guarantees absolute palindrome $\boldsymbol{\mu}^1 = {Q^\dagger}^T \boldsymbol{v}^1$.
\end{theorem}

\begin{proof}
Suppose $B$ has singular value decomposition $B = U\Sigma V^T$, then according to Theorem \ref{thm:m=n}, $A$ can be decomposed as
$$
    A = Q \Lambda Q^T = \frac{1}{2}\begin{bmatrix}
        U & U \\ V & -V
    \end{bmatrix} \begin{bmatrix}
        \Sigma & O \\ O & -\Sigma
    \end{bmatrix} \begin{bmatrix}
        U^T & V^T \\ U^T & -V^T
    \end{bmatrix},
$$
with $O \in \mathbb{R}^{n \times n}$ representing zero matrix.
Without loss of generality, we may assume that the diagonal matrix is in the form $\Sigma = \mathrm{diag}(\sigma_1,\sigma_2,\ldots,\sigma_n)$, with $\sigma_1 \le \sigma_2 \le \cdots \le \sigma_n$. Then $\Lambda = \mathrm{diag}(\sigma_1,\sigma_2,\ldots,\sigma_n,-\sigma_1,-\sigma_2,\ldots,-\sigma_n)$. We need to first make a permutation such that $\tilde{\Lambda}_{i,i} = -\tilde{\Lambda}_{2n+1-i,2n+1-i}, i = 1,\ldots, n$. This step can be well done by introducing a permutation matrix
$$
P = \begin{bmatrix}
     I_n & O \\
    O & I^{(anti)}_n
\end{bmatrix}
$$
such that $\tilde{\Lambda} = P^T \Lambda P$. The columns of orthogonal matrix $Q$ should also be switched so that the eigenpairs of $A$ are not disordered. Namely, $(n+i)^{th}$ and $(2n+1-i)^{th}$ columns need exchange, $i = 1,\ldots, n$, which can also be done by
$$
    \tilde{Q} = \frac{1}{\sqrt{2}} \begin{bmatrix}
        U & U_p \\
        V & -V_p
    \end{bmatrix} = \frac{1}{\sqrt{2}} \begin{bmatrix}
        U & U \\
        U & -V
    \end{bmatrix} P,
$$
where $U_p = U I^{(anti)}_n$ and $V_p = V I^{(anti)}_n$. Then 
$$\tilde{Q} \tilde{\Lambda} \tilde{Q}^T = (QP)(P^T \Lambda P)(P^T Q^T) = Q \Lambda Q^T = A.$$

However, the diagonal entries of $\tilde{\Lambda}$ are not sorted in a necessarily non-decreasing order. As to ensure such ordering in $\Lambda^\dagger$ based on $\tilde{\Lambda}$, it is necessary to switch $(2n+1-i)^{th}$ and $(2n+1-j)^{th}$ entries if we need to switch the $i^{th}$ and $j^{th}$ entries on the diagonal of $\tilde{\Lambda}$. Then the permutation matrix $P^\dagger$ which ensures $\Lambda^\dagger = {P^\dagger}^T \tilde{\Lambda} P^\dagger$ can be written in a block matrix form
\begin{equation}
    P^\dagger  = \begin{bmatrix}
        P^\dagger_{11} & P^\dagger_{12} \\
        P^\dagger_{21} & P^\dagger_{22}
    \end{bmatrix},
\end{equation}
where $P^\dagger_{11} = I^{(anti)}_n P^\dagger_{22} I^{(anti)}_n$ and $P^\dagger_{12} = I^{(anti)}_n P^\dagger_{21} I^{(anti)}_n$. Moreover, the columns of $\tilde{Q}$ needs reordering through right-multiplying $P^\dagger$ so that the columns of $Q^\dagger = \tilde{Q}P^\dagger$ are the eigenvectors corresponding to the non-decreasing eigenvalues of $A$ stored in $\Lambda^\dagger$, as shown by 
\begin{equation}
    Q^\dagger \Lambda^\dagger {Q^\dagger}^T = (\tilde{Q} P^\dagger) ({P^\dagger}^T \tilde{\Lambda} P^\dagger) ({P^\dagger}^T \tilde{Q}^T) = \tilde{Q} \tilde{\Lambda} \tilde{Q}^T = A.
\end{equation}

Recall that the goal is to find initial vectors $\boldsymbol{v}^1$ that guarantee the absolute palindrome of measure vectors $\boldsymbol{\mu}^1 = {Q^\dagger}^T \boldsymbol{v}^1$. Equivalently, we need to write the `absolute' version of $\boldsymbol{\mu}^1$, i.e., $\boldsymbol{\mu}^{1(abs)}$, in the form of \eqref{wabs}. We commence with writing
$$ 
\begin{aligned}
    {Q^\dagger}^T = (\tilde{Q}P^\dagger)^T &= {P^\dagger}^T \tilde{Q}^T\\
    &= \frac{1}{\sqrt{2}} \begin{bmatrix}
        {P^\dagger_{11}}^T & {P^\dagger_{21}}^T \\
        {P^\dagger_{12}}^T & {P^\dagger_{22}}^T
    \end{bmatrix} \begin{bmatrix}
        U^T & V^T \\
        U_p^T & -V_p^T
    \end{bmatrix} \\
    &= \frac{1}{\sqrt{2}} \begin{bmatrix}
        {P^\dagger_{11}}^T U^T + {P^\dagger_{21}}^T U_p^T & {P^\dagger_{11}}^T V^T - {P^\dagger_{21}}^T V_p^T \\
        {P^\dagger_{12}}^T U^T + {P^\dagger_{22}}^T U_p^T & {P^\dagger_{12}}^T V^T - {P^\dagger_{22}}^T V_p^T
    \end{bmatrix}.
\end{aligned} $$
One may prove that
$$
\begin{aligned}
    I_n^{(anti)}({P^\dagger_{11}}^T U^T + {P^\dagger_{21}}^T U_p^T) & = I_n^{(anti)}{I_n^{(anti)}}^T {P^\dagger_{22}}^T {I_n^{(anti)}}^TU^T \\ 
    & \quad + I_n^{(anti)}{I_n^{(anti)}}^T {P^\dagger_{12}}^T {I_n^{(anti)}}^T U_p^T \\
    & = {P^\dagger_{22}}^T U_p^T + {P^\dagger_{12}}^T U^T,
\end{aligned}
$$
which means for the first $n$ columns of ${Q^\dagger}^T$, $i^{th}$ row exactly equals $(2n+1-i)^{th}$ row. Similarly one may deduce
$$
\begin{aligned}
    I_n^{(anti)}({P^\dagger_{11}}^T V^T - {P^\dagger_{21}}^T V_p^T) & = I_n^{(anti)}{I_n^{(anti)}}^T {P^\dagger_{22}}^T {I_n^{(anti)}}^T V^T \\
    & \quad -I_n^{(anti)}{I_n^{(anti)}}^T {P^\dagger_{12}}^T {I_n^{(anti)}}^T V_p^T \\
    & = {P^\dagger_{22}}^T V_p^T - {P^\dagger_{12}}^T V^T \\
    & = - ({P^\dagger_{12}}^T V^T - {P^\dagger_{22}}^T V_p^T),
\end{aligned}
$$
leading to the equivalence of $i^{th}$ row and negative $(2n+1-i)^{th}$ row in ${Q^\dagger}^T$'s last $n$ columns. For simplicity we may write ${Q^\dagger}^T$ as
$$ {Q^\dagger}^T = \begin{bmatrix}
    {Q^\dagger_{11}}^T & {Q^\dagger_{21}}^T \\
    I_n^{(anti)} {Q^\dagger_{11}}^T & -I_n^{(anti)} {Q^\dagger_{21}}^T
\end{bmatrix} $$
and denote $\boldsymbol{v}^1 \in \mathbb{R}^{2n}$ by $\boldsymbol{v}^1 = \begin{bmatrix}
    \boldsymbol{v}_u \\ \boldsymbol{v}_d
\end{bmatrix}$,
with $\boldsymbol{v}_u, \boldsymbol{v}_d \in \mathbb{R}^n$. Then $\boldsymbol{\mu}^1$ reads 
$$
\begin{aligned}
    \boldsymbol{\mu}^1 = {Q^\dagger}^T \boldsymbol{v}^1 & = \begin{bmatrix}
    {Q^\dagger_{11}}^T & {Q^\dagger_{21}}^T \\
    I_n^{(anti)} {Q^\dagger_{11}}^T & -I_n^{(anti)} {Q^\dagger_{21}}^T
\end{bmatrix} \begin{bmatrix}
        \boldsymbol{v}_u \\
        \boldsymbol{v}_d
    \end{bmatrix} \\
    & =  \begin{bmatrix}
        {Q^\dagger_{11}}^T \boldsymbol{v}_u + {Q^\dagger_{21}}^T \boldsymbol{v}_d \\
        I_n^{(anti)}({Q^\dagger_{11}}^T \boldsymbol{v}_u - {Q^\dagger_{21}}^T \boldsymbol{v}_d)
    \end{bmatrix}.
\end{aligned}
$$
Now one can take $\boldsymbol{v}_u = \boldsymbol{0} \in \mathbb{R}^n$ or $\boldsymbol{v}_d = \boldsymbol{0} \in \mathbb{R}^n$ to ensure $\boldsymbol{\mu}^{1(abs)}$ of form \eqref{wabs}, indicating that $\boldsymbol{\mu}^1$ is an absolute palindrome.
\end{proof}

\subsubsection{Case 2: $n_1 > n_2$}
\label{sec:case2}
Assume $B$ is tall, skinny and of full column rank, one can prove that
\begin{theorem}
    \label{prop:2}
    Let $B \in \mathbb{R}^{n_1 \times n_2}$ be of full column rank and $A = \begin{bmatrix}
        O & B \\
        B^T & O
    \end{bmatrix} \in \mathbb{R}^{(n_1+n_2)\times (n_1+n_2)}$ be a Jordan-Wielandt matrix with eigen-decomposition $Q^\dagger \Lambda^\dagger {Q^\dagger}^T$. Assume $\Lambda^\dagger$ has non-decreasing diagonal entries. Then any initial vector $\boldsymbol{v}^1$ that has form
    $ \boldsymbol{v}^1_d = \begin{bmatrix}
        \boldsymbol{0} \\
        \boldsymbol{v}_d
    \end{bmatrix}$
    with real vector $\boldsymbol{v}_d \in \mathbb{R}^{n_2}$ and zero vector $\boldsymbol{0}\in\mathbb{R}^{n_1}$ guarantees absolute palindrome $\boldsymbol{\mu}^1 = {Q^\dagger}^T \boldsymbol{v}^1$.
\end{theorem}
\begin{proof}
    Since $B$ is of full column rank, it can be decomposed as 
    $$ B = U\Sigma V^T = \begin{bmatrix}
        U_1 & U_2
    \end{bmatrix}\begin{bmatrix}
        \Sigma_1 \\ O
    \end{bmatrix} V_1^T,$$
    where $U_1 \in \mathbb{R}^{n_1 \times n_2}$ and $V_1 \in \mathbb{R}^{n_2\times n_2}$ stores the left and right singular vectors of $B$ and $\Sigma_1$ is the diagonal matrix of $B$'s singular values. Then $A$ can be written as
    $$ A = Q\Lambda Q^T = \frac{1}{2} \begin{bmatrix}
        U_1 & U_1 & \sqrt{2}U_2 \\
        V_1 & -V_1 & O
    \end{bmatrix}\begin{bmatrix}
        \Sigma_1 & O & O \\
        O & -\Sigma_1 & O \\
        O & O & O
    \end{bmatrix} \begin{bmatrix}
        U_1^T & V_1^T \\
        U_1^T & -V_1^T \\
        \sqrt{2}U_2^T & O
    \end{bmatrix}. $$
    Similarly we can find a permutation matrix 
    $$ P = \begin{bmatrix}
        I_{n_2} & O \\
        O & I_{n_1}^{(anti)}
    \end{bmatrix} $$
    such that $\tilde{\Lambda} = P^T \Lambda P$ and $\tilde{\Lambda}_{i,i} = \tilde{\Lambda}_{n_1+n_2+1-i,n_1+n_2+1-i}$. Moreover, 
    $$\tilde{Q} = QP = \frac{1}{\sqrt{2}}\begin{bmatrix}
        U_1 & \sqrt{2}U_{2p} & U_{1p} \\
        V_1 & O & -V_{1p}
    \end{bmatrix},$$ 
    where $U_{1p} = U_1 I_{n_2}^{(anti)}$, $V_{1p} = V_1 I_{n_2}^{(anti)}$, $U_{2p} = U_2 I_{n_1-n_2}^{(anti)}$ such that $A = \tilde{Q}\tilde{\Lambda}\tilde{Q}^T$. Similar to the proof of \textbf{Theorem \ref{prop:1}}, one should adjust the columns of $\tilde{Q}$ and the orders of diagonal entries in $\tilde{\Lambda}$ via another permutation matrix $P^\dagger$, i.e., $Q^{\dagger} = \tilde{Q}P^\dagger$, $\Lambda^\dagger = {P^\dagger}^T \tilde{\Lambda} P^\dagger$, so that the diagonal entries are non-decreasing. Then ${Q^\dagger}^T$ has the form
    $$ {Q^\dagger}^T = \begin{bmatrix}
        {Q^\dagger_{11}}^T & {Q^\dagger_{21}}^T \\
        {Q^\dagger_{12}}^T & O \\
        I_n^{(anti)}{Q^\dagger_{11}}^T & -I_n^{(anti)}{Q^\dagger_{21}}^T
    \end{bmatrix}. $$
    Now if the first $n_1$ entries of an initial vector $\boldsymbol{v}^1$ are zeros, namely $\boldsymbol{v}^1 = \begin{bmatrix}
        \boldsymbol{0} \\
        \boldsymbol{v}_d
    \end{bmatrix}$ ($\boldsymbol{0} \in \mathbb{R}^{n_1}, \boldsymbol{v}_d \in \mathbb{R}^{n_2}$), then $\boldsymbol{\mu}^1 = {Q^\dagger}^T \boldsymbol{v}^1$ is an absolute palindrome.
    
\end{proof}

\subsubsection{A simple example}
A simple test is conducted to reveal the symmetry of Ritz values of a $6\times 6$ Jordan-Wielandt matrix through the construction of initial vectors in Theorem \ref{prop:1}. Let
$$ B= \begin{bmatrix}
    1 & 2 & 3 \\
    1 & 2 & 4 \\
    1 & 3 & 4
\end{bmatrix}, A = \begin{bmatrix}
    O & B \\
    B^T & O
\end{bmatrix},$$
and $\boldsymbol{v}_u^1 = (1, 1, 1, 0, 0, 0)^T$, $\boldsymbol{v}_d^1 = (0, 0, 0, 1, 1, 1)^T$, $\boldsymbol{v}_1 = (1, 1, 1, 1, 1, 1)^T$ be three initial vectors for the $4$-step Lanczos process. The Ritz values for these three cases are recorded in Table \ref{tab:simple}.

\begin{table}[htbp]
\caption{A simple test on Ritz values of $A$ generated by different starting vectors}
\label{tab:simple}
\begin{tabular}{|c|c|c|c|c|c|}
\hline
                     & $\theta_1$ & $\theta_2$ & $\theta_3$ & $\theta_4$ & Symmetric Ritz values? \\ \hline
$\boldsymbol{v}_u^1$ & -7.7838    & -0.2612    & 0.2612     & 7.7838     & Yes                    \\ \hline
$\boldsymbol{v}_d^1$ & -7.7838    & -0.2792    & 0.2792     & 7.7838     & Yes                    \\ \hline
$\boldsymbol{v}^1$   & -7.7836    & -0.3895    & 0.2293     & 7.7838     & No                     \\ \hline
\end{tabular}
\end{table}

One may observe that with an unsuitable initial vector which does not follow the construction described in Theorem \ref{prop:1} (e.g., $\boldsymbol{v}^1$ in this test), the corresponding Ritz values would exhibit asymmetry.

\section{Computation complexity for symmetric and asymmetric quadrature nodes}
\label{sec:comparison}
Given a symmetric positive definite matrix $A$, a normalized starting vector $\boldsymbol{v}$ and error tolerance $\epsilon$, the theoretical lower bounds of the Lanczos iteration $m$ that guarantee
\begin{equation}
    \label{eq:epsilon}
|Q(\boldsymbol{v},f,A) - Q_m(\boldsymbol{v},f,A)| = |\mathcal{I} - \mathcal{I}_m| \le \epsilon
\end{equation}
are different for the case of symmetric quadrature nodes and the asymmetric one (see \cite[Section 4.1]{UCS17} and \cite[Section 3]{CK21}). Note that the error of $|\mathcal{I} - \mathcal{I}_m|$ by \cite{UCS17,CK21} is analyzed based on a Bernstein ellipse $E_\rho$ with foci $\pm 1$ and elliptical radius $\rho > 1$, which is formed via an analytic continuation admitted by function $g:[-1,1] \to \mathbb{R}$, while an affine transform $h: [-1,1] \to [\lambda_{\min},\lambda_{\max}]$ is required to ensure $g = f(h(\cdot))$. In detail, the radius $\rho$ is a function of matrix $A$'s condition number $\kappa = \lambda_{\max}/\lambda_{\min} > 1$. 

Since the choice of $m$ determines the computational complexity $\mathcal{O}(n^2m)$, it is important to study the disparity between the two theoretical bounds of $m$ proposed by Theorem \ref{thm:UCSthm4.2} \cite[Theorem 4.2]{UCS17} and Theorem \ref{thm:CK21thm3} \cite[Corollary 3]{CK21} for symmetric and asymmetric Ritz values. To commence, let us provide a brief overview of these theories.
\begin{theorem}
    \label{thm:UCSthm4.2} \cite[Theorem 4.2]{UCS17}
    Let $g$ be analytic in $\left[-1, 1\right]$ and analytically continuable in the open Bernstein ellipse $E_{\rho}$ with foci $\pm1$ and elliptical radius $\rho > 1$. Let $M_\rho$ be the maximum of $|g(t)|$ on $E_\rho$. Then the $m$-point Lanczos quadrature approximation satisfies
    \begin{equation}
        \label{eq:UCSthm4.2}
        |\mathcal{I}-\mathcal{I}_m| \le \frac{4M_\rho}{1 - \rho^{-2}} \rho^{-2m}.
    \end{equation}
\end{theorem}
\begin{theorem}
    \label{thm:CK21thm3} \cite[Corollary 3]{CK21}
    Let $g$ be analytic in $\left[-1, 1\right]$ and analytically continuable in the open Bernstein ellipse $E_{\rho}$ with foci $\pm1$ and elliptical radius $\rho > 1$. Let $M_\rho$ be the maximum of $|g(t)|$ on $E_\rho$. Then the $m$-point Lanczos quadrature approximation with asymmetric quadrature nodes satisfies
    \begin{equation}
        \label{eq:CK21thm3}
        |\mathcal{I} - \mathcal{I}_m| \le \frac{4M_\rho}{1-\rho^{-1}} \rho^{-2m}.
    \end{equation}
\end{theorem}

Note that the elliptical radius $\rho$ is not fixed. In this paper we set 
\begin{equation}
    \label{eq:rho}
    \rho = \frac{\lambda_{\max}+\lambda_{\min}}{\lambda_{\max} - \lambda_{\min}}  = \frac{\sqrt{\kappa} + 1}{\sqrt{\kappa} - 1}.
\end{equation}
Then combining \eqref{eq:epsilon}, \eqref{eq:CK21thm3} and \eqref{eq:UCSthm4.2} we have
\begin{equation}
    \label{eq:m_asym}
    m_{\text{asym}} \ge \frac{1}{2 \log(\rho)}\cdot \left[ \log(4M_{\rho}) - \log(1-\rho^{-1}) - \log(\epsilon) \right],
\end{equation}
\begin{equation}
        \label{eq:m_sym}
        m_{\text{sym}} \ge \frac{1}{2 \log(\rho)}\cdot \left[ \log(4M_{\rho}) - \log(1-\rho^{-2}) - \log(\epsilon) \right].
\end{equation}
One may observe that the two bounds only differs from the term 
$$ m^* = m_{\text{asym}} - m_{\text{sym}} = \log(1+\rho^{-1})/2\log(\rho).$$ Corollary \ref{coro:diffm} studies the lower and upper bounds of such difference.
\begin{corollary}
\label{coro:diffm}
    Let $A\in \mathbb{R}^{n \times n}$ be a real symmetric positive definite matrix with eigenvalues in $[\lambda_{\min},\lambda_{\max}]$, $f$ be an analytic function on $[\lambda_{\min},\lambda_{\max}]$, $\boldsymbol{u}$ be a normalized vector, $\kappa = \lambda_{\max}/\lambda_{\min}$ be the condition number. Then the difference $m^*$ of iterations in the cases of asymmetric and symmetric quadrature nodes should satisfy
    \begin{equation}
        \label{eq:diffm}
        \frac{(\sqrt{\kappa}-1)^2}{8\sqrt{\kappa}} \le m^* \le \frac{\sqrt{\kappa} - 1}{4}
    \end{equation}
    to guarantee the same computation precision of $Q_m(\boldsymbol{u},f,A)$.
\end{corollary}
\begin{proof}
    As mentioned, the difference between \eqref{eq:m_asym} and \eqref{eq:m_sym} lies in the term $m^* = \log(1+\rho^{-1})/2\log(\rho)$, so the inequalities \eqref{eq:diffm} result from analyzing the lower and upper bounds of it. Note that 
    $$ 1-\frac{1}{x} \le \log(x) \le x-1, \text{ for all }x>0,$$    $$ \frac{x}{1+x} \le \log(1+x) \le x, \text{ for all } x > -1.$$
    Since $\rho = (\sqrt{\kappa}+1)/(\sqrt{\kappa}-1) > 1$, the lower bound becomes
    $$ \frac{\log(1+\rho^{-1})}{2\log(\rho)} \ge \frac{\frac{1}{\rho + 1}}{2(\rho - 1)} = \frac{1}{2(\rho^2 - 1)} = \frac{(\sqrt{\kappa} - 1)^2}{8\sqrt{\kappa}},$$
    and the upper bound reads
    $$ \frac{\log(1+\rho^{-1})}{2\log(\rho)} \le \frac{\frac{1}{\rho}}{2(1-\frac{1}{\rho})} = \frac{1}{2(\rho - 1)} = \frac{\sqrt{\kappa} - 1}{4}.$$
\end{proof}

One may observe that when matrix becomes ill-conditioned (i.e., $\kappa$ becomes larger), the difference between two bounds of $m$ in asymmetric and symmetric cases increases. We make further discussions on Corollary \ref{coro:diffm} in Section \ref{sec:test_diffm}. 

\section{Numerical experiments}
\label{sec:experiments}
\subsection{Test on Theorem \ref{Thm:main}}
\label{sec:TestThmMain}
We conduct experiments to verify whether Theorem \ref{Thm:main} is valid. The first three cases study the reproducible matrix $A = H \Lambda H^T \in \mathbb{R}^{n\times n}$ with the diagonal matrix 
$$\Lambda = \mathrm{diag}(\lambda_1, \lambda_2, \cdots, \lambda_n),$$
and the Householder matrix is constructed according to \cite{zhu09}
$$ H = I_n - \frac{2}{n}(\mathbf{1}_n \mathbf{1}^T_n), $$
where $I_n$ represents the identity matrix of size $n$ and  $\mathbf{1}_n$ denotes an $n$-dimensional vector of all-ones. The eigenvalues stored in $\Lambda$ are predetermined and $n = 50$ is fixed. The fourth case focuses on the \texttt{nd3k} matrix from \cite{DH11}.
\begin{itemize}
    \item[-] Case 1: $\{\lambda_i\}_{i=1}^{50} = \{i/50\}_{i=1}^{50}$, $\boldsymbol{v} = \mathbf{1}_{50}/\sqrt{50}$;
    \item[-] Case 2: $\{\lambda_i\}_{i=1}^{50} = \{1/(51-i)\}_{i=1}^{50}$, $\boldsymbol{v} = \mathbf{1}_{50}/\sqrt{50}$;
    \item[-] Case 3: $\{\lambda_i\}_{i=1}^{50} = \{i/50\}_{i=1}^{50}$, $\boldsymbol{v} = (1,2,\cdots,50)^T/\Arrowvert (1,2,\cdots,50)^T \Arrowvert$;
    \item[-] Case 4: \texttt{nd3k} matrix, $\boldsymbol{v} = (1,\cdots,1,-1,\cdots,-1)^T/\sqrt{9000} \in \mathbb{R}^{9000}.$
\end{itemize}
Details of these 4 cases with different eigenvalues and initial vectors can be seen in Table \ref{Tab:4cases}. 
\begin{table}[htbp]
\caption{Details of 4 cases with different eigenvalue distributions and starting vectors}
\label{Tab:4cases}
\begin{center}
    \begin{tabular}{|c|c|c|c|c|}
\hline
       & Case 1 & Case 2 & Case 3 & Case 4\\ \hline
Do $A$ have symmetric eigenvalues? & Yes & No & Yes & No     \\ \hline
Is $\boldsymbol{\mu}_1$  an absolute palindrome?  & Yes & Yes & No & No    \\ \hline
Are Ritz values symmetric? & Yes & No & No & No \\ \hline
Any sufficient conditions? & Yes & ? & ? & ? \\ \hline
\end{tabular}
\end{center}

\end{table}
Recall that the figure of $\mu(t)$ would be central symmetric about $(\bar{\lambda}, \mu(\bar{\lambda}))$ if $A$ has symmetric eigenvalue distribution and $\boldsymbol{\mu}_1$ is an \emph{absolute palindrome}. Figure \ref{Fig:Test} showcases the measure function $\mu(t)$ alongside the corresponding $10$ Ritz values for the four cases, offering a visual depiction that substantiates the validity of Theorem \ref{Thm:main} to a certain degree.

Indeed, it is important to acknowledge that while Theorem \ref{Thm:main} serves as a sufficient condition guaranteeing the symmetry of Ritz values produced by the Lanczos iteration, it is not strictly necessary. Researchers may delve deeper to explore whether alternative conditions exist that render these Ritz values symmetric in cases $2,3,4$.

\begin{figure}[htbp]
    \centering
    \subfloat[Case 1: symmetric Ritz values]{\label{Fig:case1}\includegraphics[width = 0.5\textwidth]{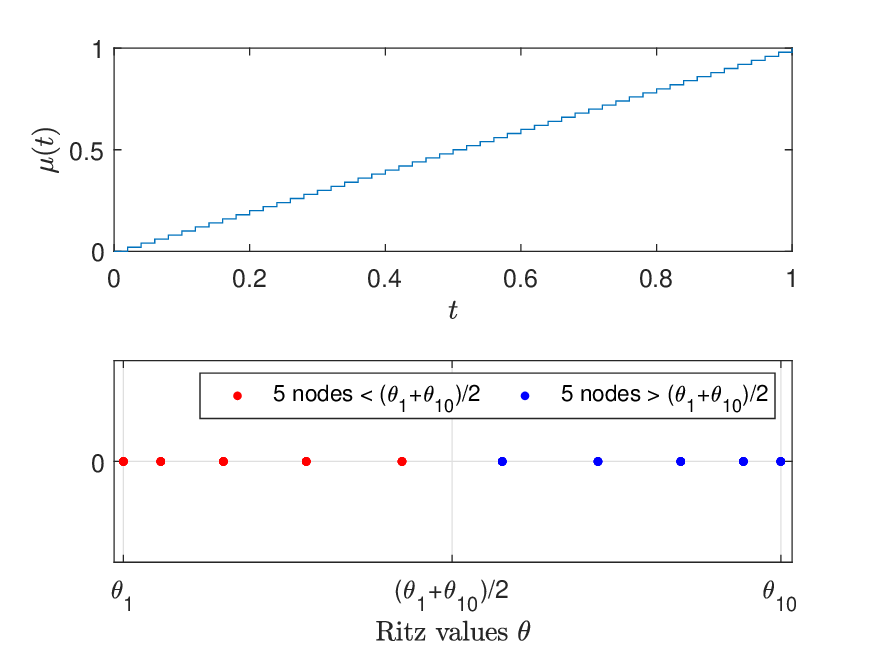}}
    \subfloat[Case 2: asymmetric Ritz values]{\label{Fig:case2}\includegraphics[width = 0.5\textwidth]{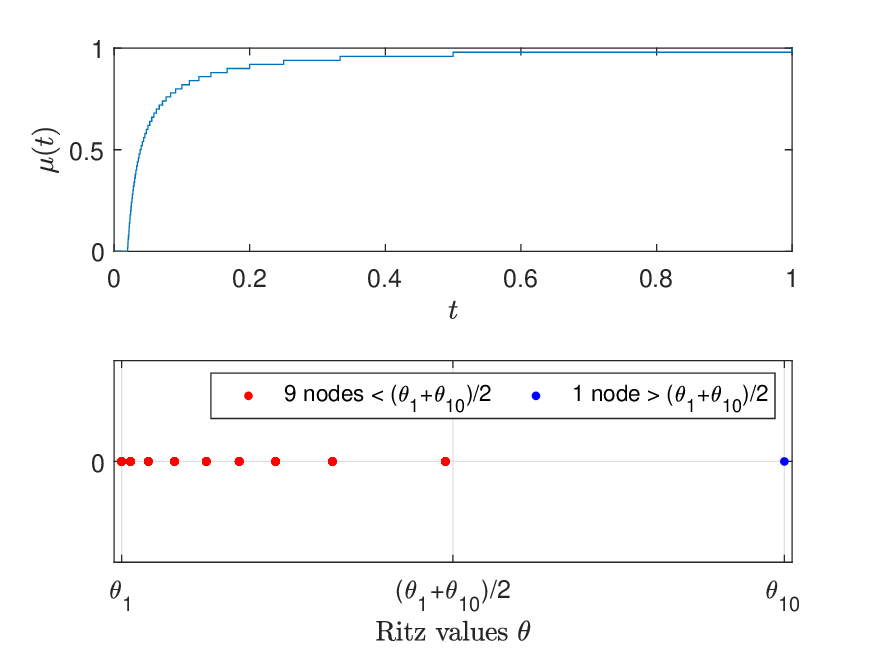}}\\
    \subfloat[Case 3: asymmetric Ritz values]{\label{Fig:case3}\includegraphics[width = 0.5\textwidth]{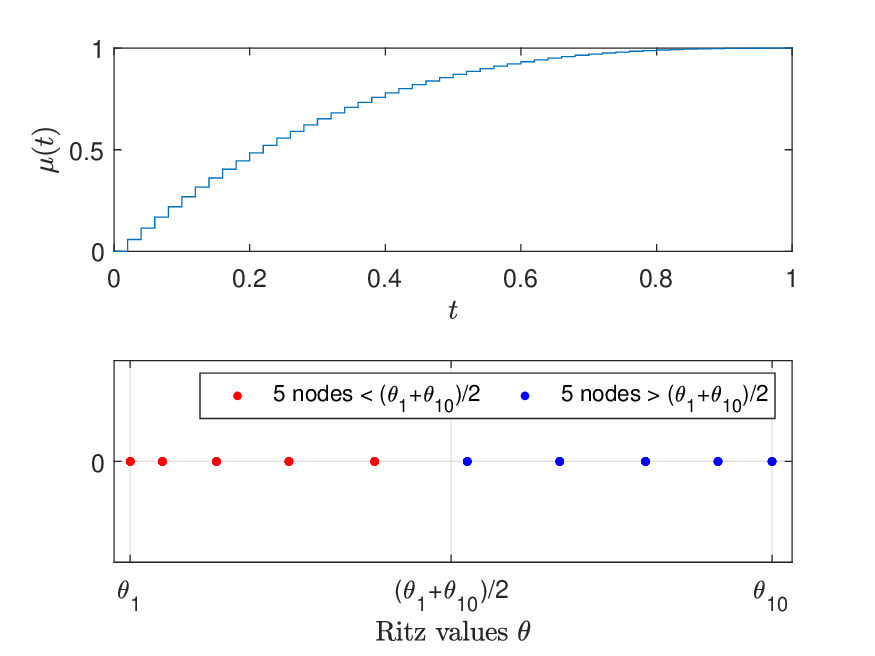}}
    \subfloat[Case 4: asymmetric Ritz values]
    {\label{Fig:case4}\includegraphics[width = 0.5\textwidth]{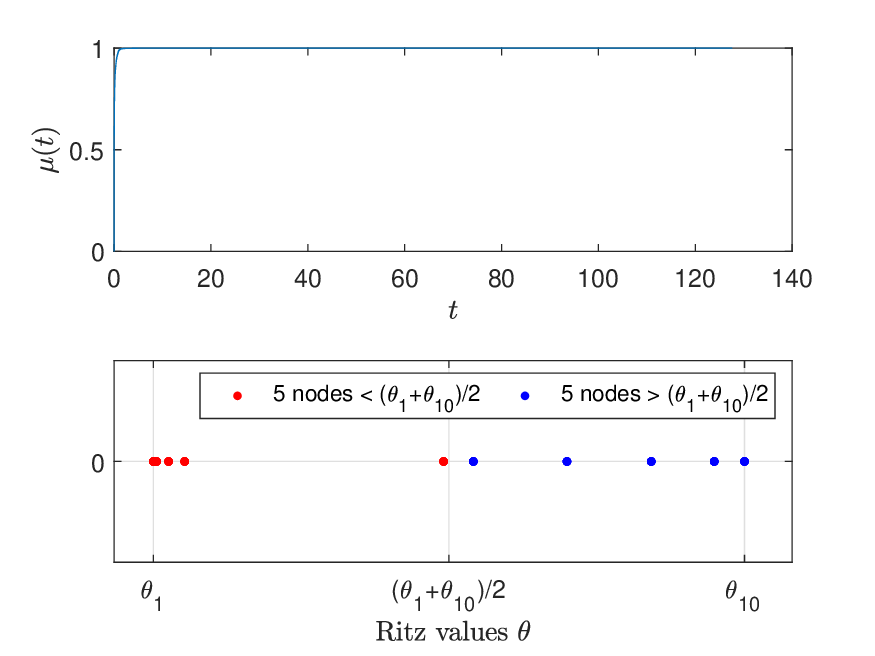}}
    \caption{Plots of discrete measure functions \eqref{eq:measure_f} and the locations of Ritz values in four cases.}
    \label{Fig:Test}
\end{figure}

\subsection{Test on Theorem \ref{prop:1} in applications of complex network}
\label{sec:estrada}
We show that one can construct starting vectors as in Theorem \ref{prop:1} to ensure the symmetry of quadrature in applications such as complex network. More specifically, in applications of estimating Estrada index (EI), such construction would reduce the variance of estimation and not violate the unbiasedness.

In the analysis of layer-coupled multiplex networks, a directed network of $n$ nodes with asymmetric adjacency matrix $B \in \mathbb{R}^{n\times n}$ can be extended to a bipartite undirected network with symmetric \textit{block supra-adjacency matrix} of form \eqref{Eq:A} via bipartization \cite{BEK13,BB20,BS22}. The Estrada index, namely 
\begin{equation}
    \label{eq:estrada}
    EI(A,\beta) = \mathrm{tr}(e^{\beta A}) = \sum_{i=1}^{2n} \boldsymbol{e}_i^T e^{\beta A} \boldsymbol{e}_i,
\end{equation}
provides scalar measures for the connectivity of the full network. The classical Hutchinson trace estimator \cite{H90} uses the relation
\begin{equation}
    \label{eq:Hutchinson}
    \mathrm{tr}(f(A)) = \sum_{i=1}^{2n}[f(A)]_{ii} \approx \frac{1}{N}\sum_{k=1}^N \boldsymbol{z}_k^T f(A) \boldsymbol{z}_k =  \frac{2n}{N}\sum
_{k=1}^N \frac{\boldsymbol{z}_k^T}{\sqrt{2n}}f(A) \frac{\boldsymbol{z}_k}{\sqrt{2n}},
\end{equation}
where the entries of $\boldsymbol{z}_k \in \mathbb{R}^{2n}$ follow the Rademacher distribution, i.e., every entry can take $+1$ or $-1$ with probability $1/2$ each. Gaussian trace estimator or normalized Rayleigh-quotient trace estimator also helps approximate $\mathrm{tr}(f(A))$ in \eqref{eq:Hutchinson} \cite{AT11}. Such quadratic terms in \eqref{eq:Hutchinson} can be approximated by the Gaussian quadrature \cite{UCS17}. Based on Theorem \ref{prop:1}, one may generate initial vectors with half Rademahcer entries ($\pm 1$) and half zeros (denoted by \textit{half-zero-half-Rademacher} vectors in this context) in the form of $\boldsymbol{v}_u^1$ or $\boldsymbol{v}_d^1$ in Section \ref{sec:case1} and Section \ref{sec:case2} to guarantee symmetry quadrature nodes and unbiased estimators for \eqref{eq:estrada}. Recall that 
\begin{equation}
    \label{eq:hutch2}
    \sigma^2 \mathrm{tr}(f(A)) = \mathbb{E}[\boldsymbol{z}^T f(A)\boldsymbol{z}],
\end{equation}
where the entries of $\boldsymbol{z}$ follows a distribution $\mathcal{Z}$ of mean zero and variance $\sigma^2$ \cite{H90}. For half-zero-half-Rademahcer initial vectors, $\mathcal{Z}$ has mean zero and variance $1/2$ since 
$$ \mathbb{E}[\mathcal{Z}] = \frac{1}{4}\cdot 1 + \frac{1}{4} \cdot (-1) + \frac{1}{2} \cdot 0 = 0, $$
$$ \sigma^2 = \mathbb{V}[\mathcal{Z}] = \frac{1}{4} \cdot (1-0)^2 + \frac{1}{4}\cdot (-1-0)^2 + \frac{1}{2}\cdot (0-0)^2 = \frac{1}{2}. $$

Numerical experiments are conducted to show that half-zero-half-Rademacher initial vectors can reduce the variance of estimating $\mathrm{tr}(e^{\beta A})$ by the stochastic Lanczos quadrature method \cite{UCS17}, compared with using fully Rademacher distributed initial vectors. 

\subsubsection{Test on synthetic matrix}
We first study a synthetic randomly generated Jordan-Wielandt matrix with $B = U\Sigma V^T \in \mathbb{R}^{1000\times 1000}$. $U$ and $V$ are obtained by generating Gaussian distributed matrices and then orthogonalizing them, i.e., 
$$\texttt{U = orth(randn(1000,1000));} \texttt{V = orth(randn(1000,1000));}$$
The diagonal matrix $\Sigma$ (storing $B$'s singular values) is also randomly generated, 
$$\texttt{Sigma = diag(randn(1000,1));}$$
Then the tested synthetic matrix is created by 
$$\texttt{A = [zeros(1000,1000) U*Sigma*V'; V*Sigma*U' zeros(1000,1000)];}$$
In our test, one normalized Rademacher vector and two normalized half-zero-half-Rademacher vectors are randomly generated (but with reproducible seed generator) as the initial vectors for 100 times. We compare the variances of these estimators for $\mathrm{tr}(e^{\beta A})$ with $\beta = 1$ and 100 Lanczos iterations. As shown in Table \ref{table:syn} and Figure \ref{fig:syn}, half-zero-half-Rademacher initial vectors have lower variances than the one of Rademacher distributed vectors within the same computational budget. 

\begin{figure}[htbp]
    \centering
    \includegraphics[width = 0.7 \textwidth]{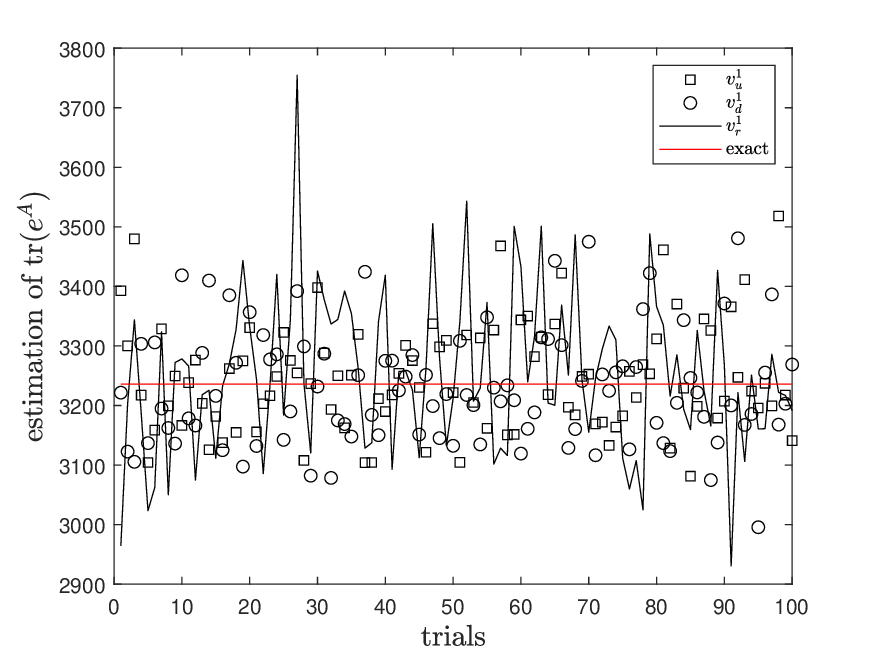}
    \caption{100 trials of stochastic Lanczos quadrature estimators of $\mathrm{tr}(e^{\beta A})$ with $\beta = 1$, $m=100$, synthetic bipartite matrix $A$ and different initial vectors. $\boldsymbol{v}^1_u$ and $\boldsymbol{v}_d^1$ are half-zero-half-Rademacher vectors with either upper or lower part of entries taking $\pm 1$, while all elements of $\boldsymbol{v}^1_r$ are Rademacher distributed.}
    \label{fig:syn}
\end{figure}

\begin{table}[htbp]
\caption{Variances of stochastic Lanczos quadrature estimators in Figure \ref{fig:syn}.}
\label{table:syn}
\begin{center}
    \begin{tabular}{@{}llll@{}}
\toprule
                            & $\boldsymbol{v}_u^1$ & $\boldsymbol{v}_d^1$ & $\boldsymbol{v}_r^1$ \\ \midrule
Variance of estimations & 90.86                             & 97.08                             & 134.57                          \\ \bottomrule
\end{tabular}
\end{center}
\end{table}

\subsubsection{Test on \textit{email-Eu-core-temporal} data set}
Then a directed network example of 1005 nodes and 24929 edges from \textit{email-Eu-core-temporal} data set \cite{PBL17,DH11} is tested. The adjacency matrix $B$ of this directed network is non-symmetric, so a Jordan-Wielandt matrix in the form of \eqref{Eq:A} is built and the estimation of Estrada index $\mathrm{tr}(e^{\beta A})$ with parameter $\beta = 0.5/\lambda_{\max}$ is of interest. Similar to the results of synthetic test, Table \ref{table:email} and Figure \ref{fig:email} also reflect the effect of variance reduction by utilizing half-zero-half-Rademacher vectors in the Lanczos process.

\begin{figure}[htbp]
    \centering
    \includegraphics[width = 0.7 \textwidth]{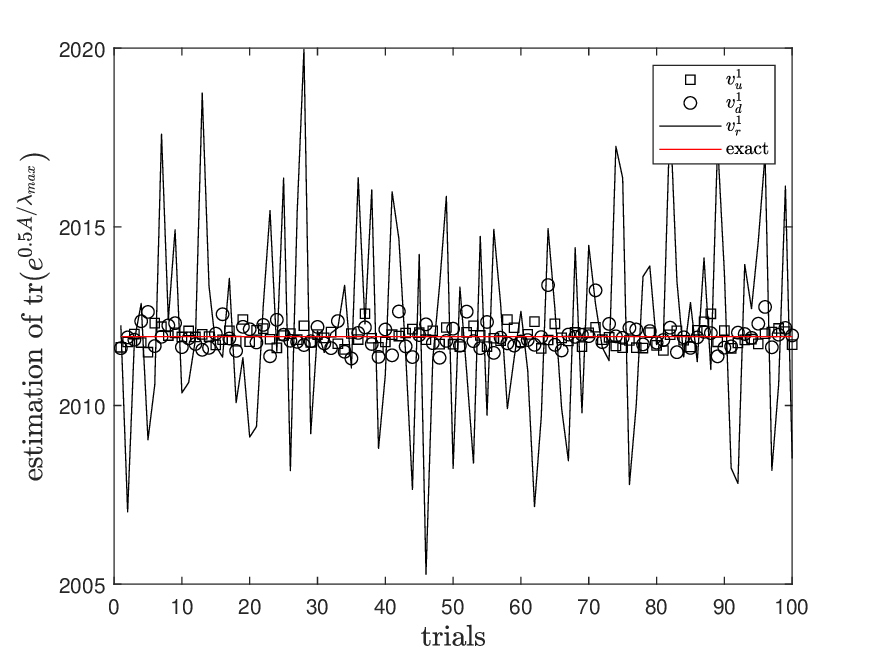}
    \caption{100 trials of stochastic Lanczos quadrature estimators of $\mathrm{tr}(e^{\beta A})$ with $\beta = 0.5/\lambda_{\max}$, $m=100$, bipartite matrix based on \textit{email-Eu-core-temporal} data set \cite{PBL17,DH11} and different initial vectors. $\boldsymbol{v}^1_u$ and $\boldsymbol{v}_d^1$ are half-zero-half-Rademacher vectors with either upper or lower part of entries taking $\pm 1$, while all elements of $\boldsymbol{v}_r^1$ are Rademacher distributed.}
    \label{fig:email}
\end{figure}

\begin{table}[htbp]
\caption{Variances of stochastic Lanczos quadrature estimators in Figure \ref{fig:email}.}
\label{table:email}
\begin{center}
    \begin{tabular}{@{}llll@{}}
\toprule
                            & $\boldsymbol{v}_u^1$ & $\boldsymbol{v}_d^1$ & $\boldsymbol{v}_r^1$ \\ \midrule
Variance of estimations & 0.22                             & 0.32                             & 2.69                          \\ \bottomrule
\end{tabular}
\end{center}
\end{table}

\subsection{Test on Corollary \ref{coro:diffm}}
\label{sec:test_diffm}
We conduct another test on the difference $m^*$ between the numbers of  Lanczos iterations that ensure the same precision of approximating quadratic forms \eqref{eq:GQ} based on asymmetric and symmetric quadrature nodes. Figure \ref{fig:diff} depicts the lower and upper bounds of $m^*$, and the average of two bounds along with the exact $m^*$ for different condition numbers $\kappa = 10,50,100,500,\ldots,5\times10^4,10^5$. The elliptical radius is chosen as $\rho = (\sqrt{\kappa} + 1)/(\sqrt{\kappa} -1)$.

\begin{figure}
    \centering
    \includegraphics[width = 0.7\textwidth]{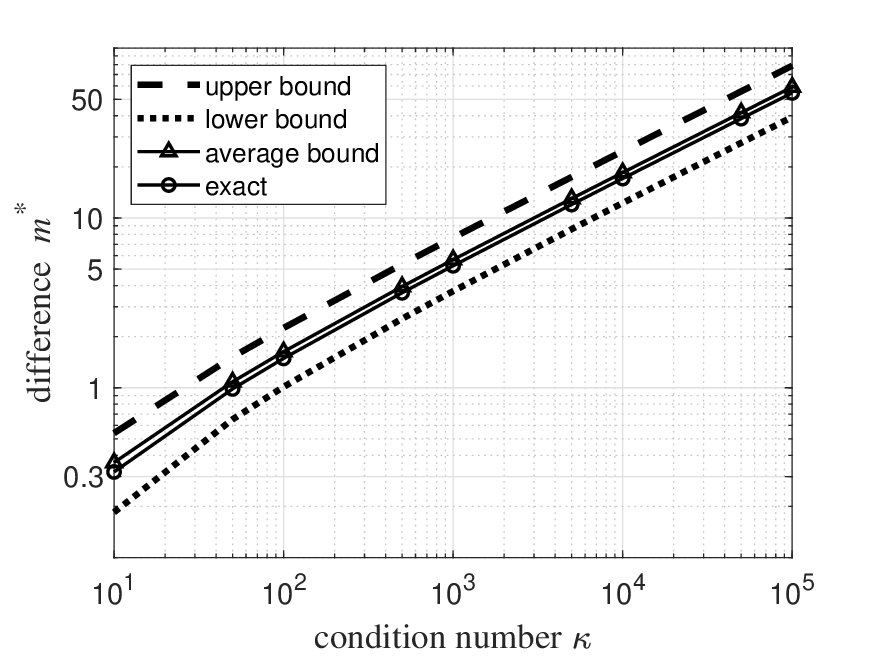}
    \caption{Lower, upper bounds, the average of the two bounds, and the exact values of the absolute difference $m^*$ between the Lanczos iterations required for equivalent accuracy in approximating quadratic forms \eqref{eq:GQ} with asymmetric and symmetric quadrature nodes, where $\kappa = 10,50,100,500,\ldots,5\times 10^4,10^5$.}
    \label{fig:diff}
\end{figure}
It is observed that the exact difference $m^*$ can be roughly substituted by the average of its lower and upper bounds in \eqref{eq:diffm}. When $\kappa$ becomes large, e.g., $\kappa = 10^5$, the gap of Lanczos iterations is around $55$, which leads to huge computational costs for extremely large matrices.

\section{Concluding remarks}

\label{sec:cr}

For a long time the matrix-based quadrature rule generated via the Golub-Welsch algorithm has been assumed and used as symmetric. It is only recently that doubts have begun to be cast into account about the existence of such symmetric quadrature rules for practical applications. This paper proves that with a delicate choice of the starting vector in the Lanczos process, a wide class of the Jordan-Wielandt matrices can realize symmetric quadrature rules. Note that the study on the existence of such symmetric quadrature rule for general matrices requires further discussion. According to this analysis, such a general realization of symmetry is probably both problem dependent and initial vector dependent.




\vspace{0.5cm}

\thanks{\textbf{Acknowledgements:} We are grateful to Professor Zhongxiao Jia of Tsinghua University and Professor James Lambers of University of Southern Mississippi for critical reading of the manuscript and providing useful feedback. We also appreciate the feedback provided by the reviewers.}

\bibliographystyle{amsplain}
\bibliography{ref.bib}

\end{document}